\documentclass[11pt]{article}
\usepackage{amssymb}
\usepackage{latexsym}
\usepackage{amsthm}
\usepackage{amscd}
\usepackage{amsmath}
\usepackage{indentfirst}
\usepackage[pdftex]{graphicx}

\theoremstyle{definition}
\newtheorem{thm}{Theorem}[section]
\newtheorem{lem}[thm]{Lemma}
\newtheorem{th-def}[thm]{Theorem-Definition}
\newtheorem{cor}[thm]{Corollary}
\newtheorem{defn-lem}[thm]{Definition-Lemma}

\newtheorem{prop}[thm]{Proposition}

\newtheorem{rem}[thm]{Remark}

\numberwithin{equation}{section}

\allowdisplaybreaks[4]

\def \Q{{\mathbb Q}}

\def \C{{\mathbb C}}

\def \Z{{\mathbb Z}}
\def \R{{\mathbb R}}

\def\map#1.#2.{#1 \longrightarrow #2}
\def\rmap#1.#2.{#1 \dasharrow #2}
\DeclareMathOperator{\rank}{rank}

\DeclareMathOperator{\Spec}{Spec}

\DeclareMathOperator{\Hom}{Hom}

\DeclareMathOperator{\Emb}{Emb}
\DeclareMathOperator{\Arth}{Arth}

\def\fb#1.{\underset #1 \to \times}
\def\pr#1.{\Bbb P^{#1}}
\def\ring#1.{\mathcal O_{#1}}
\def\mlist#1.#2.{{#1}_1,{#1}_2,\dots,{#1}_{#2}}

\def\Hom{\operatorname{Hom}}

\def\uloopr#1{\ar@'{@+{[0,0]+(-4,5)} @+{[0,0]+(0,10)}
@+{[0,0]+(4,5)}}
  ^{#1}}
\def\dloopr#1{\ar@'{@+{[0,0]+(-4,-5)} @+{[0,0]+(0,-10)}
@+{[0,0]+(4,-5)}}
  _{#1}}

\def\rloopd#1{\ar@'{@+{[0,0]+(5,4)} @+{[0,0]+(10,0)}
@+{[0,0]+(5,-4)}}
  ^{#1}}
\def\lloopd#1{\ar@'{@+{[0,0]+(-5,4)} @+{[0,0]+(-10,0)}
@+{[0,0]+(-5,-4)}}
  _{#1}}

\long\def\ignore#1{}
\long\def\ignore#1{#1}
\title{An Enriques involution of a supersingular K3 surface over odd characteristic}
\author{Junmyeong Jang}
\date{}

\begin{document}

\maketitle

\medskip
     \section{Introduction}
     Let $k$ be an algebraically closed field. A projective minimal smooth surface $Y$ over $k$ is an Enriques surface if $K_{Y}^{2}=0$ and
     $\dim H^{2}_{\acute{e}t}(Y, \Q _{l})=10$. We assume the characteristic of the base field $k$ is not 2 and $Y$ is an Enriques surface over $k$.
     It is known that $H^{1}(Y, \mathcal{O}_{Y})= H^{2}(Y,\mathcal{O}_{Y})=0$ and $K_{Y}$ is non-trivial divisor. Since $K_{Y}$ is
     a line bundle of order 2, $Y$ has a double \'{e}tale cover $X \to Y$
     such that $\mathcal{O}_{X} = \mathcal{O}_{Y} \oplus K_{Y}$.
     Here $X$ is a K3 surface. Since the (local or global) moduli space of algebraic K3 surfaces is of
     19 dimension and the moduli space of Enriques surfaces is of 10 dimension, not every K3 surface is the double cover of an Enriques surface.
     Over the field of complex numbers, $\C$ the criterion for a K3 surface to be the double \'{e}tale cover of an Enriques surface
     comes from the Torelli theorem for K3 surfaces.
      We put $\Gamma$ is the unique even unimodular $\Z$-lattice of rank 10 of
     signature (1,9). $\Gamma$ has a decomposition $ \Gamma = U \oplus E_{8}$, where $U$ is the hyperbolic lattice of rank 2 and $E_{8}$ is the negative definite lattice induced by the root system $E_{8}$.
      The torsion free Neron-Severi lattice of an Enriques surface is isomorphic to $\Gamma$. The criterion is following.

       \vspace{0.2cm}
     \noindent
     \textbf{Theorem} (\cite{Ke})
    A complex K3 surface $X$ is the double cover of an Enriques surface or equivalently, $X$ has a fixed point free involution if and only if
      there exists a primitive embedding $\Gamma (2) \hookrightarrow NS(X)$ such that the orthogonal complement of $\Gamma (2)$ in $NS(X)$ has no vector of self intersection $-2$.

     \vspace{0.2cm}

      Note that here the orthogonal complement of $\Gamma (2)$ in $NS(X)$ is even negative definite.
      But over a field of positive characteristic, of course, the Torelli theorem is meaningless.
       However, if the characteristic of $k$ is $p>2$ and $X$ is a supersingular $K3$ surface, we have the crystalline Torelli theorem.
       (\cite{Og2})
     In this paper, using the crystalline Torelli theorem, we prove that for a supersingular K3 surface
     over an algebraically closed filed of characteristic $p>2$, the same criterion to be the double cover of an Enriques surface as
     the complex case holds.

     \vspace{0.2cm}
     \noindent
     \textbf{Theorem 4.1.}
    Assume $X$ is a superingular K3 surface defined over an algebraically closed field of characteristic $p>2$. $X$ is an Enriques K3 surface if and only if
     there exists a primitive embedding of $\Gamma (2) = U (2) \oplus E_{8}(2)$ into $NS(X)$ such that the orthogonal complement
     of $\Gamma(2)$ in $NS(X)$ does not contain a vector of self intersection $-2$.

     \vspace{0.2cm}

     This theorem depend only on the lattice structure of the Neron-Severi group. The Neron-Severi lattice of a supersingular K3 surface  is determined only by
     the characteristic of the base field and the Artin invariant. (see section 2) Therefore the criterion to have an Enriques involution
     is applied to the family of supersingular K3 surfaces of the same Artin invariant.

     Over an odd characteristic, a supersingular K3 surface of Artin invariant 1 is the Kummer sufrace associated to the product of two supersingular
     elliptic curves. Hence it has an Enriques involution. Using the above theorem, we prove over almost all characteristic, a supersingular K3 surface
     has an Enriques involution if and only if the Artin invariant is less than 6.

     \vspace{0.2cm}
     \noindent
     \textbf{Corollary 4.7.}
    Let $k$ be an algebraically closed field of characteristic $p=19$ or $p>23$ and $X$ be a supersingular K3 surface over $k$. $X$ is an Enriques K3 surface
     if and only if the Artin invariant of $X$ is less than 6.   \\[0.6cm]
    {\bf Acknowledgment}\\
    The author appreciates to J.Keum and M.Sch\"{u}tt for helpful comments and correcting mistakes.
    This research was supported by Basic Science Research Program through the National Research Foundation of Korea(NRF) funded by the Ministry of Education, Science and Technology(2011-0011428).

     \section{F-Crystals of K3 surfaces}
     In this section, we review some properties of crystalline cohomologies of K3 surfaces and the crystalline
     Torelli theorem of supersingular K3 surfaces.

     Let $k$ be an algebraically closed field of characteristic $p>2$,
     $W= W(k)$ be the ring of Witt vectors of $k$ and $K$ be the fraction field of $W$.
     Assume $X$ is a K3 surface defined over $k$. $X$ does not have a non-vanishing 1-form (\cite{RS},\cite{N1}), so
     the Hodge diagram of $X$ is
     $$
     \begin{array}{ccc}
     1 &0&1\\
     0& 20 & 0\\
     1 & 0 & 1.
     \end{array}
     $$
     Hence the Hodge-Derham spectral sequence of $X/k$ degenerates at $E_{1}$ and
     $$H^{1}_{dr}(X/k) = H^{3}_{dr}(X/k)=0.$$
     Considering the crystalline universal coefficient theorem (\cite{BO}, 7.26)
     $$ 0 \to H^{i} _{cris}(X/W) \otimes _{W} k \to H^{i} _{dr}(X/k) \to Tor ^{1}_{W} ( H^{i+1} _{cris}(X/W), k) \to 0,$$
     $H^{1}_{cris}(X/W) = H^{3}_{cris}(X/W)=0$ and $H^{2}_{cris} (X/W)$ is a free $W$-module of rank 22. $H^{i}_{cris}(X/W)$ is equipped
     with a canonical Frobenius linear endomorphism induced by the absolute Frobenius morphism of $X$. We denote this morphism by
     $$\mathbf{F} : H^{i}_{cris}(X/W) \to H^{i}_{cris}(X/W).$$
     The rational
     crystalline cohomology $H^{i}_{cris} (X/K) = H^{i}_{cris}(X/W) \otimes _{W} K$ with the induced Frobenius linear morphism $\mathbf{F}:
     H^{i}_{cris}(X/K) \to H^{i}_{cris}(X/K)$ is an F-isocrystal over $k$. An F-isocrystal is determined by its Newton slopes up to isomorphism as follow.
     Let $\sigma : K \to K$ be the Frobenius morphism on $K$. Put $K[T]$ is a Frobenius commutative polynomial ring over $K$ such that
     $Ta = \sigma (a) T$ for any $ a\in K$. Then we have a non-canonical isomorphism
     $$ H^{i} _{cris} (X/K) \simeq \bigoplus _{j} K[T] / (T^{r_{j}} - p^{s_{j}})$$
     of F-isocrystal. Here $r_{j}$ is a positive integer and $s_{j}$ is a non-negative integer. In the given isomorphism,
     $\mathbf{F}$ corresponds to the multiplication by $T$. The Newton slopes of $H^{i}_{cris}(X/K)$ are the
     rational number $s_{j}/r_{j}$ with multiplicity $r_{j}$.
     A refined information of Newton slopes of the crystalline cohomology can be obtained using
     the Derham-Witt complex
     $$0 \to W\mathcal{O}_{X} \to W\Omega ^{1}_{X} \to W\Omega ^{2}_{X} \to 0.$$
     We recall some properties of the Derham-Witt complex from \cite{I1}.
     The crystalline cohomology is realized as the hyper-cohomology
     of the Derham-Witt complex. The naive filtration of the Derham-Witt complex give raise to the
     slope spectral sequence
     $$ H^{i}(X,W\Omega ^{j} _{X}) \Longrightarrow H^{i+j} _{cris}(X/W).$$
     Here the sheaf cohomology $H^{i}(X, W\Omega _{X} ^{j})$ is not a finitely generated $W$-module in general,
     but the rank of its free part is finite and the spectral sequence
     $$H^{i}(X,W\Omega _{X} ^{j}) \otimes K \Longrightarrow H^{i+j}_{cris}(X/K)$$
     degenerates at $E_{1}$ level.
     The sheaf $W \Omega ^{j} _{X}$ has two operators $F$ and $V$ which
     are Frobenius $W$-linear and inverse-Frobenius $W$-linear respectively satisfying $FV=VF=p$. Hence $H^{i}(X, W\Omega ^{j}_{X}) \otimes _{W} K$
     with the induced operator $F$ is an F-isocrystal. Here $FV=p$ and $V$ is topologically nilpotent, so the Newton slopes of
     $(H^{i}(X,W\Omega ^{j}_{X}), F)$ are contained in $[0,1)$.
     \begin{prop}
          [\cite{I1}, p.616] The Newton slopes of $(H^{i}(X,W\Omega ^{j}_{X}), p^{j}F)$ are the $[j,j+1)$ part of
     the Newton slopes of $H^{i+j}_{cris}(X/K)$.
     \end{prop}
     By the Poincare duality, the Newton slopes of $H^{2}_{cris}(X/K)$ are symmetric. In other words,
     for $H^{2}_{cris}(X/K)$,
     the multiplicity of the slope $\alpha$ $(0\leq \alpha \leq 2))$ is equal to the multiplicity of the
     slope $2-\alpha$. Since $\dim _{K} H^{2}_{cris}(X/K)=22$ is fixed, the Newton slopes of $H^{2}(X/K)$
     are determined by the Newton slopes of $H^{2}(X,W\mathcal{O}_{X})\otimes _{W} K$.
     On the other hand, $H^{2}(X, W\mathcal{O}_{X})$ is the Dieudonn\'{e} module of the
     formal Brauer group of $X$, $\widehat{Br}_{X}$. (\cite{AM}) Since $H^{1}(X,\mathcal{O} _{X})=0$ and
     $\dim _{k} H^{2}(X, \mathcal{O}_{X})=1$, $\widehat{Br}_{X}$ is a 1-dimensional smooth formal group and
     the Dieudonn\'{e} module of $\widehat{Br} _{X}$ is
     \begin{center}

     $H^{2}(X,W\mathcal{O}_{X})$ = $ W[F,V]/(FV=p, F=V^{h-1})$, $h\geq 1$ \hspace{0.1cm} or
     \end{center}
     \begin{center}
          $H^{2}(X,W\mathcal{O}_{X})$      = $ k[[V]]$, $F=0$.
     \end{center}
     In the first case, we say the height of
     $X$ is $h$ and in the second case, we say the height of $X$ is $\infty$. If the height is $h < \infty$,
     $H^{2}(X,W\mathcal{O}_{X})$ is a free $W$-module of rank $h$ and the all the Newton slopes are $1-1/h$.
     In this case $1- 1/h$ is the only slope of $H^{2}_{cris}(X/K)$ less than 1 and the multiplicity of $1-1/h$ is $h$.
     Also the only Newton slope greater than 1 is $1+1/h$ with multiplicity $h$. If $h=1$, we say $X$ is ordinary.
     If $X$ is ordinary, the Newton polygon and the Hodge polygon of $X$ coincide.
     If the height is $\infty$,
     $H^{2}(X, W\mathcal{O}_{X}) \otimes _{W} K =0$ and all the slopes of $H^{2}_{cris}(X/K)$ are 1.

     Now consider the Kummer
     sequence on the flat topology of $X$,
     $$ 0 \to \mu _{p^{n},X} \to \mathbb{G} _{m,X} \overset {p^{n}}{\to} \mathbb{G} _{m,X} \to 0.$$
     From the associated long exact sequence, we obtain the short exact sequence
     $$ 0 \to NS(X) \otimes \Z/p^{n} \Z \to H^{2}(X, \mu _{p^{n},X}) \to Br(X)[p^{n}] \to 0.$$\
     Since all the terms of the above sequence are of finite length, they satisfy the Mittag-Leffler condition and
     after the inverse limit of the sequence, we obtain an exact sequence
     $$ 0 \to NS(X) \otimes \Z _{p} \to H^{2}(X, \Z _{p}(1)) \to T_{p}(Br(X)) \to 0.$$
     Since $NS(X) \otimes \Z _{p} \hookrightarrow H^{2}(X, \Z _{p}(1))$, the Picard number of
     $X$ is equal to or less than the rank of the free part of $H^{2}(X, \Z _{p}(1))$. By the way,
     there exists an exact sequence
     $$0 \to H^{2}(X, \Z_{p}(1)) \to H^{1}(X, W\Omega ^{1}_{X}) \overset{F-1}{\to} H^{1}(X, W\Omega ^{1} _{X}) \to 0. \ \cite{N1}$$
     Since $H^{1}(X, W\Omega ^{1}_{X}) \otimes _{W} K$ is the slope $[1,2)$-part of F-isocrystal $H^{2}_{cris}(X/K)$ and on
     $H^{1}(X, W\Omega ^{1}_{X}) \otimes _{W} K$, $\mathbf{F} = pF$, from the above we obtain
     $$0 \to H^{2}(X, \Z_{p}(1)) \otimes \Q _{p} \to H^{2}_{cris}(X/K) \overset{\mathbf{F}-p}{\to} H^{2}_{cris}(X/K) \to 0.$$
     The kernel of the right map is $\Q _{p}$-space whose dimension is the multiplicity of slope 1 of $H^{2}_{cris}(X/K)$.
     Therefore the Picard number of $X$ is equal to or less than the multiplicity of slope 1 of $H^{2} _{cris}(X/K)$.
     Since the Picard number of $X$ is positive, the multiplicity of slope 1
     of $H^{2}_{cris}(X/K)$ should be positive.
     If the height of $X$ is $h < \infty$, the multiplicity of
     slope 1 is  $22-2h>0$,
     so $h \leq 10$.
     The Newton polygon and the Hodge polygon of a K3 surface $X$ are following.
    \begin{center}
     \includegraphics[scale=0.7]{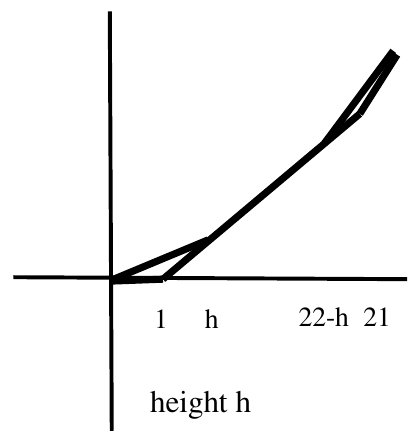}
     \hspace{2.5cm}
     \includegraphics[scale=0.7]{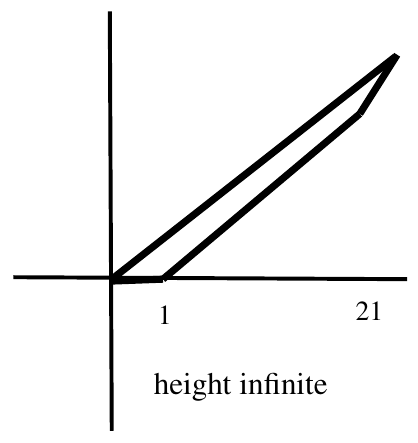}
    \end{center}
     Note that here the Newton polygon lies above the Hodge Polygon.(\cite{BO}, 8.39)

     When the height of $X$ is $\infty$, we say $X$ is Artin-supersingular. If the Picard number of $X$
     is maximal, 22, we say $X$ is supersingular. By the above argument, if $X$ is supersingular, $X$ is Artin-supersingular.
     It is conjectured that every Artin-supersingular K3 surface $X$ is supersingular.
     If the base field $k$ is a finite field this conjecture is equivalent to the Tate conjecture for Artin-supersingular K3 surfaces.
     The Tate conjecture holds for a K3 surface if the characteristic of the base field is $p\geq5$. (\cite{N2},\cite{NO}, \cite{C})
     It follows that if the base field is of characteristic $p \geq 5$, an Artin-supersingular K3 surface is supersingular by a deformation argument. (cf. \cite{A})
     It is also known that an Artin supersingular
      K3 surface equipped with an elliptic fibration is supersingular. (\cite{AS}) If the Picard number of a K3 surface $X$ is grater than 4, the Neron-Severi lattice $NS(X)$ has an isotropic vector, so it has an elliptic fibration. Hence an Artin-supersingular K3 surface with a Picard number greater than 4 is supersingular.

      When $X$ is a supersingular K3 surface, $NS(X)$ is a lattice of rank 22
     and the inclusion
     $$NS(X) \otimes \Z_{p} \hookrightarrow H^{2}(X, \Z _{p}(1)) \hookrightarrow H^{1} (X, W\Omega ^{1}_{X})$$
     gives
     an isomorphism
     $$NS(X) \otimes W \simeq H^{1}(X, W\Omega ^{1}_{X}). \ (\cite{N1},\textrm{ p.520})$$ This isomorphism preserves the lattice structures.
     The signature of $NS(X)$ is $(1,21)$ and the discriminant of $NS(X)$ is $-p^{2\sigma}$ with $1 \leq \sigma \leq 10$. We call this value $\sigma$ the Artin-invariant of $X$. It is known that the isomorphism class of the lattice $NS(X)$ is determined by 
     $\sigma$. (\cite{RS2})
     We denote this lattice
     by $\Lambda _{p,\sigma}$. It is also known that supersingular K3 surfaces of Artin invariant $\sigma$ forms $(\sigma -1)$-dimensional family
     and a supersingular K3 surface of Artin invariant 1 is unique up to isomorphism.
     For an explicit construction of $\Lambda _{p, \sigma}$, see \cite{Sh1}.
      \begin{prop} \label{sec2}
      When $X$ is a supersingular K3 surface over $k$, there exists a
     surjection $ H^{2}_{cris}(X/W)/(NS(X) \otimes W) \to H^{2}(X, \mathcal{O}_{X})$.
     \end{prop}
     \begin{proof}
     For a projective smooth surface, the slope spectral sequence degenerates at
     $E_{2}$ level. Moreover the only possibly non-trivial map on $E_{1}$ level of the slope spectral sequence is $H^{2}(X,W\mathcal{O}_{X}) \to H^{2}(X,W\Omega ^{1} _{X})$. (\cite{I1}, p.619)
     When $X$ is a supersingular K3 surface $H^{2}_{cris}(X/K)$ does not have slope 2, so $H^{0}(X, W\Omega ^{2}_{X})=0$. It follows that $F^{1} H^{2} _{cris}(X/W)
     = NS(X) \otimes W$, here $F^{\cdot}$ means the filtration given by the slope spectral sequence. Therefore
     $$H^{2}_{cris}(X/W)/( NS(X) \otimes W)
     = \ker d: H^{2}(X, W\mathcal{O}_{X}) \to H^{2}(X, W\Omega ^{1}_{X}).$$
     Considering the discriminant, the length of $H^{2}_{cris}(X/W)/(NS(X) \otimes W)$ is $\sigma$.
     Since $H^{3}_{cris}(X/W)=0$, $d:H^{2}(X,W\mathcal{O}_{X}) \to H^{2}(X, W\Omega ^{1}_{X})$ is surjective and the length of $\ker d$ is $\sigma$. Therefore
     $$\ker d = H^{2}_{cris} (X/W) / (NS(X) \otimes W) = k[[V]]/ V ^{\sigma}.$$
     By the way, from the exact sequence
     $$ 0 \to W\mathcal{O}_{X} \overset {V} {\to} W\mathcal{O}_{X} \to \mathcal{O}_{X} \to 0,$$
     we have $H^{2}(X, \mathcal{O}_{X}) = H^{2}(X, W\mathcal{O}_{X}) /V$. Therefore
     there exists a canonical morphism $H^{2}_{cirs}(X/W)/(NS(X) \otimes W) \to H^{2}(X,\mathcal{O}_{X})$. \hspace{0.5cm}
     \end{proof}

     If the height of a K3 surface $X$ is $h < \infty$, $H^{2}(X, W\mathcal{O} _{X})$ is a finite free $W$-module, so the slope spectral sequence of $X$ degenerates
     at $E_{1}$ level. (\cite{I1})
     And since the Newton polygon of the F-crystal $(H^{2}_{cris}(X/W), \mathbf{F})$ touches the Hodge polygon
     at the two edge points, the F-crystal $H^{2}_{cris}(X/W)$ has a decomposition
     \begin{center}
     $ H^{2}_{cris}(X/W) = E_{1-1/h} \oplus E_{1} \oplus E_{1+1/h}$. (\cite{Ka}, 1.6.1)
     \end{center}
     Here F-crystal $E_{1-1/h}= (H^{2}(X, W\mathcal{O}_{X}), F)$ and $E_{1+1/h} = \mathfrak{Hom}(E_{1-1/h}, W[2])$.
     The F-crystal $E_{1}$ has constant Newton slopes 1, and the Hodge polygon of $E_{1}$ coincides with its Newton polygon.
     Hence $E_{1}$ is $p$-times of a unit crystal. The cup product induces a perfect paring on $H^{2}_{cris}(X/W)$ which is compatible with
     $\mathbf{F}$. For this paring, $E_{1-1/h}$ and $E_{1+1/h}$ are isotropic and
     dual to each other. $E_{1}$ with the restricted pairing is a unimodular $W$-lattice.
     Put $T = H^{2}(X, \Z _{p}(1)) = E_{1} ^{\mathbf{F}-p}$. $T$ is a
     free $\Z _{p}$ module of rank $22-2h$. $T$ with the induced paring is a $\Z _{p}$-unimodular lattice. A $\Z _{p}$-unimodular lattice
     is determined up to isomorphism by the rank and the discriminant. The discriminant of $T$ = $- \left( \frac{-1}{p} \right)$. (\cite{Og2}, p.363) Here $\left(
     \frac{\cdot}{\cdot} \right)$ is the Legendre symbol. We conclude that if the height of $X$ is $h< \infty$, the F-crystal $H^{2} _{cris}(X/W)$
     with the lattice structure is determined by $h$ and it has no more information. However, if $X$ is supersingular, this information actually determines the isomorphism class of $X$.
     \begin{thm}[\cite{Og2}, Crystalline Torelli theorem]
     $X$ and $Y$ are supersingular K3 surface defined over $k$. Let $\Psi : H^{2}_{crius}(X/W) \to H^{2} _{cris}(Y/W)$ be an isomorphism
     compatible with the Frobenius morphism and the cup product. Then $X$ and $Y$ are isomorphic.
     \end{thm}
     For a supersingular K3 surface $X$, the $\Z _{p}$-lattice $NS(X) \otimes \Z _{p}$ has a decomposition
     $$ NS(X) \otimes \Z _{p} = E_{0} \oplus E_{1}.$$
     Here $E_{1}$ is a unimodular $\Z _{p}$-lattice of rank $22-2\sigma$ and
     $E_{0}$ is the $p$ times of a unimodular $\Z _{p}$-lattice of rank $2\sigma$.
     Recall that $H^{2}_{cris}(X/W) \otimes k \simeq H^{2} _{dr}(X/k)$. Composing with $NS(X) \hookrightarrow H^{2}_{cris}(X/W)$,
     we have the induced morphism $\lambda : NS(X) \otimes k \to H^{2}_{dr}(X/k)$. We call $K_{X} = \ker \lambda \subseteq NS(X) \otimes k$
     the period space of a supersingular K3 surface $X$.
     $K_{X}$ is a $\sigma$-dimensional isotropic subspace of $E_{0} \otimes k$.
      Now we state another form of crystalline Torelli theorem which we will use later in this paper.
     \begin{thm}[\cite{Og2}, p.371]
     $X$ and $Y$ are supersingular K3 surfaces defined over $k$. $\Psi : NS(X) \to NS(Y)$ is an isometry. If $\Psi$ takes
     the ample cone of $NS(X) \otimes \R$ into the ample cone of $NS(Y) \otimes \R$ and $K_{X}$ onto $K_{Y}$, then
     $\Phi$ is induced by an isomorphism $\psi : Y \to X$
     \end{thm}

     \section{Lattices}
     In this section, we recall some facts on lattices.

     We denote an even unimodular $\Z$-lattice of rank 2 with signature $(1,1)$
     $$ \left(
     \begin{array}{cc}
     0 &1 \\ 1 & 0
     \end{array} \right)$$
     by $U$.
     We denote the Cartan matrix of the root system $E_{8}$ by $E_{8}$ again.
     $E_{8}$ is a negative definite even unimodular $\Z$-lattice of rank 8.
     An even unimodular lattice of signature $(s,t)$ exists if and only if
     $s-t \equiv 0$ (mod 8) and it is a direct sum of finite number of $U$ and $\pm E_{8}$.
     In particular, it is unique up to isomorphism. When $Y$ is an Enriques surface over
     an algebraically closed field $k$, the numerical lattice of $Y$ is a unimodular even lattice
     of signature (1,9). (\cite{CD}, p.117) It is isomorphic to $U \oplus E_{8}$. We denote this lattice by $\Gamma$ and call
     the Enriques lattice.
     If $X$ is a K3 surface defined over $\C$, the singular cohomology $H^{2}(Y,\Z)$ is a free abelian group
     of rank 22 equipped with an even unimodular lattice structure of signature $(3,19)$. Hence the lattice
     $H^{2}(X,\Z)$ is isomorphic to $U^{\oplus 3} \oplus E_{8}^{\oplus 2}$. We denote this lattice by $\Lambda$ and call
     the K3 lattice.

     Let $L$ be an even $\Z$-lattice and $L^{*}= \Hom (L , \Z)$ be the dual of $L$. There is a natural
     embedding $L \hookrightarrow L^{*}$ and the cokernel is a finite abelian group. We call this finite group the discriminant group
     of $L$ and denote by $l(L)$.
     The order of $l(L)$ is the absolute value of $d(L)$, the discriminant of $L$. The discriminant group $l(L)$
     is equipped with a non-degenerated bilinear form $b_{L} : x,y \mapsto (x,y) + \Z$ with values in $\Q /\Z$ and with the associated
     quadratic form $q_{L} : x \mapsto (x,x) + 2\Z$ with values in $\Q / 2\Z$.

     Let $M \hookrightarrow L$ be an inclusion of $\Z-$lattices of same rank. We say $L$ is  an over lattice of $M$.
     Then there is a chain of embeddings of lattices of same rank $M \subset L \subset L^{*} \subset M^{*}$. We denote $S(L) = L/M \subset l(M)$.
     Then $S(L)$ is an isotropic subgroup of $l(M)$ and $S(L) ^{\perp} = L^{*}/M$ by definition.
     \begin{prop}[\cite{Ni}, p.110]
     The correspondence $L \mapsto S(L)$ gives a bijection between the over lattices of $M$ and the isotropic subgroups of $l(M)$.
     \end{prop}
      We assume $M \hookrightarrow L$ is a primitive embedding of even lattices and
     $n= \rank L$, $r =\rank M$. $(r \leq n)$ $M ^{\perp}$ is the orthogonal complement of $M$ in $L$. $M^{\perp}$ is
     also a primitive sublattice of $L$ and $\rank M^{\perp}= n-r$. $L$ is an over lattice of $M \oplus M^{\perp}$.
     \begin{prop}[\cite{Ni}, p.111] \label{Ni2}
     For $S(L) \subseteq l(M \oplus M^{\perp}) = l(M) \oplus l(M^{\perp})$, the projection $S(L) \to L(M)$ and $S(L) \to l(M^{\perp})$ are
     injective. Let $S_{1}$ and $S_{2}$ be the image of $S(L)$ by the above two projections and $\gamma : S_{1} \to S_{2}$ be the isomorphism
     compatible with the projections.
     Let $\phi : M \to M$ and $\psi : M ^{\perp} \to M^{\perp}$ be isometries.
     $\phi \oplus \psi$ extends to an isometry of $L$ if and only if $\bar{\phi} : l(M) \to l(M)$ and $\bar{\psi} : l(M^{\perp})
     \to l(M^{\perp})$ preserves
     $S_{1}$ and $S_{2}$ respectively and $\gamma \circ \bar{\phi} = \bar{\psi} \circ \gamma$.
     \end{prop}
     \begin{prop}[\cite{Sh3}, p.915] \label{SH1}
     Let $M \hookrightarrow L$ be a primitive embedding of even lattices. If $d(M)$ and $d(L)$ are relatively prime,
     then
     $$(l(M^{\perp}), q_{M^{\perp}}) =
     (l(M), q_{M}) \oplus (l(L), q{L}).$$
     Moreover $|S(L)| = |l(M)|$ and the projection $S(L) \to l(M)$ is bijective.
     \end{prop}
     Now we recall Shimada's result on a comparison of a primitive embedding of an even lattice into $\Lambda$ and $\Lambda _{p, \sigma}$.
     Let us consider a condition for $\Z$-lattices $L$ and $M$,
     $$\Emb (M,L) : \textrm{there exists a primitive embedding of } M \textrm{ into }L. $$
     Suppose $d$ is an integer, $p$ is a prime number which does not divide $2d$ and $\sigma$ is an positive integer less than 11.
     We consider the following condition
     $$\Arth (p, \sigma , d) : \left ( \frac{(-1)^{\sigma +1} d}{p} \right) = -1.$$
     Here $\left ( \frac{\cdot}{\cdot} \right )$ means the Legendre symbol.

     \begin{prop}[\cite{Sh2}, p.396] \label{SH2}
     Assume $M$ is an even lattice of rank $r$ and of signature $(e_{+},e_{-})$
     such that $e_{+} \leq 1$, $e_{-} \leq 19$. $d= d(M)$ is the discriminant of $M$ and
     $p$ is a prime number which does not divide $2d$.

      (1) If $r > 22-2\sigma$, $\Emb ( M, \Lambda _{p, \sigma})$ does not hold.

      (2) If $r < 22- 2\sigma$, $\Emb ( M ,\Lambda _{p, \sigma})$ is equivalent to $\Emb (M , \Lambda)$.

      (3) If $r= 22 - 2\sigma$, $\Emb ( M, \Lambda _{p, \sigma})$ is equivalent to $\Emb (M, \Lambda)$ and $\Arth ( p, \sigma , d)$.

     \end{prop}

     Let $k$ be an algebraically closed field of characteristic$\neq 2$,
     $Y$ be an Enriques surface over $k$ and $f : X \to Y$ be the K3 covering of $Y$.
     The pull back $f^{*}$ gives a primitive embedding $\Gamma (2) \hookrightarrow NS(X)$. (\cite{Na}, p.206)
     In particular the Picard number of $X$, $\rho (X) \geq 10$.
     \begin{prop}\label{J1}
     Assume the characteristic of $k$ is $p>2$. If $X$ is of finite height, the height $h(X) \leq 6$.
     If the height of $X$ is infinite, $X$ is supersingular and the Artin invariant of $X$, $\sigma (X) \leq 5$.

     \end{prop}
     \begin{proof}
     Recall that the Picard number of a smooth algebraic surface is equal to or less than the dimension of the slope 1 part
     of $H^{2}_{cris}(X/K)$. When $X$ is of finite height $h$, the dimension of the slope 1 part of $H^{2}_{cris}(X/K)= 22-2h$, so
     $h \leq 6$. Assume the height of $X$ is infinite. Since the Picard number of $X$ is greater than 4, $X$ has an elliptic fibraiton,
     so $X$ is supersingular. Let $\sigma$ be the Artin invariant of $X$.
     Since the discriminant of $\Gamma (2)$ is $-2^{10}$, by proposition \ref{SH2} (1), $\sigma \leq 6$. When $\sigma =6$, $(-1)^{5} d(\Gamma (2)) = 2^{10}$ is a perfect square, so
     $\Arth (p, \sigma , d(\Gamma (2) ))$ fails. By proposition \ref{SH2} (3), there is no primitive embedding of $\Gamma$ into $\Lambda _{p, 6}$.
     \hspace{0.5cm}
     \end{proof}
     \begin{rem} \label{rem}
     The above proposition states that for an Enriques K3 surface, there are 11 possibilities of the Frobenius invariants.
     Let us consider the moduli space of Enriques surfaces (local or global) which is 10 dimensional and the K3 cover of the universal family. Since generally in a family of K3 surfaces,
     the one step degeneration
     of Frobenius invariants is of codimension 1, it is reasonable to expect that all the possible Frobenius invariants, in the above proposition,
     actually occur for an Enriques K3 surface.
     \end{rem}

     \section{Enriques involution of a supersingular K3 surface}
     Let $k$ be an algebraically closed field of characteristic $p>2$.
     \begin{thm} \label{th1}
     Assume $X$ is a superingular K3 surface defined over $k$. $X$ is an Enriques K3 surface if and only if
     there exists a primitive embedding of $\Gamma (2) = U (2) \oplus E_{8}(2)$ into $NS(X)$ such that the orthogonal complement
     of $\Gamma(2)$ in $NS(X)$ does not contain a vector of self intersection $-2$.
     \end{thm}
     \begin{proof}
     Assume $X$ is the K3 cover of an Enriques surface $Y$ and $f: X \to Y$ is the covering map. Let
     $i : X \to X$ be the associated involution. $N=f^{*} NS(Y) \simeq \Gamma (2)$ is a primitive sublattice of $NS(X)$.
     Let $M$ be the orthogonal complement of $N$ in $NS(X)$.
     $M$ is negative definite.
     $N$ and $M$ are the eigenspace of $i ^{*}$ in $NS(X)$ for eigenvalues 1 and -1 respectively.
     Assume there is a vector $v \in M$ such that $(v,v)=-2$. Then by the Riemann-Roch theorem, either of $v$ or $-v$ is effective.
     Suppose $v$ is effective and $-v$ is not. But $i^{*}(v)=-v$ and it contradicts. Therefore $M$ does not contain
     a vector of self intersection -2.

     Conversely we assume there exists a primitive embedding $N \simeq \Gamma (2) \hookrightarrow NS(X)$ and the orthogonal complement of $N$
     in $NS(X)$, $M$ does not contain a vector of self intersection -2. Let $\psi = id \oplus -id : N \oplus M \to N \oplus M$.
     \begin{lem} $\psi$ extends to $NS(X)$.

     \end{lem}
     \begin{proof}
     Let $S = NS(X) /(N \oplus M)$ and $S_{1}$ and $S_{2}$ be the image of $S$ in $l(N)$ and $l(M)$ by the projections.
     $\gamma : S_{1} \to S_{2}$ is the isomorphism compatible with the projections.
     $\sigma$ is the Artin invariant of $X$, $d(NS(X)) = -p^{2\sigma}$ and
     $d(N)= -2^{10}$. Hence by Proposition \ref{SH1},
     \begin{center}
     $l(M) = l(N) \oplus l(NS(X)) = (\Z /2)^{10} \oplus (\Z /p) ^{2\sigma}$ and $S \simeq (\Z /2)^{10}$.
     \end{center}
     $id$ and $-id$ preserve $S_{1}$ and $S_{2}$ respectively and $\gamma \circ id = -id \circ \gamma$, so by proposition \ref{Ni2}, $\psi$ extends
     to $NS(X)$.
     \end{proof}
     We denote the extension of $\psi$ to $NS(X)$ by $\psi$ again.
     $V = NS(X) \otimes \R$ and
     $$P = \{ x \in V | (x,x)>0 \}$$ is the positive cone of $X$. $\Delta = \{ v \in NS(X) | (v,v)=-2 \}$,
     the set of roots of $X$. For any $v \in \Delta$, we denote the reflection with respect to
     $v$ by $s_{v} : u \mapsto u+ (u,v)v$. Let us denote the subgroup of the orthogonal group of $NS(X)$ generated by all the refections $s_{v}, ( v \in \Delta)$ and $-id$ by $W_{X}$.
     We set $P_{1} = \{ v \in P | (v, w) \neq 0,\ \forall w \in \Delta \}$. It is known that the $W_{X}$ acts simply transitively on the set of connected components of $P_{1}$. Moreover
     we can prove that the connected component which contains an ample divisor is the ample cone of $X$. (\cite{Og2}, \cite{Na}) In other words, $v \in NS(X)$ represents an ample divisor
     if and only if $(v,v)>0$ and $(v, w)>0$ for all $w \in \Delta$. Equivalently, if $(v,v)>0$ and $(v,w)\neq 0$ for any $w \in \Delta$, there exists a unique element $\gamma$ in $W_{X}$ such that
     $\gamma (v)$ represents an ample divisor.
     \begin{lem}
     There exists $\gamma \in W_{X}$ such that
     $\gamma \circ \psi \circ \gamma ^{-1}$ is induced by an automorphism of $X$
     \end{lem}
     \begin{proof}
     Since the signature of $N$ is $(1,9)$ and $M \cap \Delta = \varnothing$, there exists $v \in N \cap P_{1}$ such that $(v,v)>0$. By the above argument there exists
     $\gamma \in W_{X}$ such that $\gamma (v)$ is an ample class. We may assume $N$ itself contains an ample class.
     Because $\psi$ fixes an ample class, it preserves the ample cone. By the crystalline Torelli theorem, it is enough to
     show that $\psi$ preserves the period space of $X$,
     $$K_{X} = \ker (NS(X) \otimes k \to H^{2}_{cris}(X/W) \otimes k = H^{2}_{dr}(X/k)).$$
     Since $d(N) = -2^{10}$, $N \otimes W$ is unimodular in $H^{2}_{cris}(X/W)$, so $H^{2}_{cris}(X/W)$ has a decomposition
     $$H^{2}_{cris}(X/W) = (N \otimes W) \oplus M',$$
     where $M'$ is a unimodular $W$-overlattice of $M \otimes W$. We have
     $$K_{X} = \ker (M \otimes k \to H^{2}_{cris}(X/W) \otimes k).$$
     Therefore $\psi$ induces a multiplication by $-1$ on $K_{X}$ and $K_{X}$ is preserved by $\psi$.
     \end{proof}
     We assume $\gamma$ is identity and $N$ contains an ample class. Let $g : X \to X$ be the automorphism
     of $X$ which induces $\psi$ on $NS(X)$. $g$ is an involution since so is $\psi$. Now we show $g$ has no fixed point.
     We will follow the argument in \cite{Na}.
     By the above decomposition $H^{2}_{cris}(X/W) = (N \otimes W) \oplus M'$, we obtain
     $$ M'/ (M \otimes W) \simeq H^{2}_{cris} (X/W) / (NS(X) \otimes W).$$
     Hence $g$ induces $-1$ on $H^{2}_{cris} (X/W) / (NS(X) \otimes W)$ and -1 on $H^{2}(X, \mathcal{O}_{X})$ by proposition \ref{sec2}. The Serre duality implies
     $g$ induces $-1$ on $H^{0}(X, \Omega ^{1}_{X/k})$ and also on $\Omega ^{1}_{X/k}$. Assume $x \in X$ is a fixed point of $g$. Because $g$ acts on $\Omega _{X/k} ^{1}$ by -1,
     there exists an \'{e}tale neighborhood of $x$, $\Spec k[t_{1}, t_{2}]$ such that $g^{*}(t_{1}) = t_{1}$, $g^{*}(t_{2}) = -t_{2}$. Therefore
     the fixed locus of $g$ is a divisor of $X$ and $Y=X/g$ is a smooth surface. Let $f : X \to Y$ be the projection.
     Since $N \cong E(2) \oplus E_{8} (2)$ does not contain a vector with self intersection $-2$, $Y$ is a minimal surface. Let $D$ be the fixed locus of $g$ and $C= f(D)$. $C$ is the ramification locus of $f$. It follows that $f^{*}(K_{Y} + C) = K_{X} \equiv 0$ and $-K_{Y}$ is effective, so $Y$ is a rational surface. However
     $H^{2}_{\acute{e}t}(Y, \Q _{l}) = H^{2}_{\acute{e}t} (X, \Q _{l}) ^{g} = N \otimes \Q _{l}$ is of 10 dimensional. It contradicts since
     $\dim H^{2} _{\acute{e}t} (Y, \Q _{l})$ is 2 or 3. Therefore $g$ does not have a fixed point.
     $X$ is an \'{e}tale double cover of $Y$, so $Y$ is an Enriques surface and $X$ has an Enriques involution. \hspace{0.5cm}
     \end{proof}

     \begin{rem}
     Theorem \ref{th1} is only concerned with the lattice structure of $NS(X)$. The Neron-Severi lattice of a supersingular
     K3 surface is determined by the characteristic of the base field and the Artin invariant.
     Assume $k$ and $k'$ be algebraically closed fields of same characteristic $p>2$,
     if one supersingular K3-surface of Artin invariant $\sigma$ over $k$ has an Enriques involution, every supersingular K3-surface of
     Artin invariant $\sigma$ over $k'$ has an Enriques involution.
    \end{rem}

     $k$ is an algebraically closed filed of characteristic $p>2$.
     Since the unique supersingular K3 surface of Artin invariant 1 over $k$ is the Kummer surface of the product of two supersingular elliptic curves, 
     it has an Enriques involution. (cf. \cite{MN}, p.383) In the following we prove a supersingular K3 surface has an Enriques involution if and only if the Artin invariant
     is less than 6 over almost all characteristics.

     \begin{thm}\label{th2}
     Let $k$ be an algebraically closed field of characteristic $p$ and $X$ be a supersingular K3 surface of Artin invariant $\sigma$.
     \begin{description}
     \item[(1)] If $p=11$ or $p\geq 19$ and $\sigma$= 3 or 5, $X$ has an Enriques involution
     \item[(2)] If $p\geq13$, $p\neq23$ and $\sigma$=2 or 4, $X$ has an Enriques involution.
     \end{description}

     \end{thm}
     \begin{proof}
     For any primitive embedding $\Gamma (2) \hookrightarrow \Lambda$, the orthogonal complement is isomorphic to $U \oplus U(2) \oplus E_{8}(2)$. (\cite{H}, p.81)
     Let $\mathbf{x,y}$ be a standard basis of $U(2)$ such that the matrix corresponding to the basis is
     $$\left( \begin{array}{cc}
     0 &2\\ 2& 0
     \end{array} \right).$$
     We also fix a basis of $E_{8}(2)$, $\mathbf{e}_{1}, \cdots , \mathbf{e}_{8}$ for which the corresponding matrix of the lattice is
     $$\left(
     \begin{array}{rrrrrrrr}
     -4&2&0&0&0&0&0&0\\
     2&-4&2&0&0&0&0&0\\
     0&2&-4&2&2&0&0&0\\
     0&0&2&-4&0&0&0&0\\
     0&0&2&0&-4&2&0&0\\
     0&0&0&0&2&-4&2&0\\
     0&0&0&0&0&2&-4&2\\
     0&0&0&0&0&0&2&-4
     \end{array} \right).
      $$
     For a positive integer $d$, we put
     \begin{center}
     $M_{2,d} =
     \left(
     \begin{array}{rr}
     4d &0 \\ 0&-4 \end{array} \right)$,
     $N_{5,d} =
     \left(
     \begin{array}{rr}
     -4d &0 \\ 0&-4 \end{array} \right)$,
     \end{center}
     \begin{center}
     $M_{3,d} =  \left(
     \begin{array}{rrrr}
     4d &0&0&0\\ 0&-4&0&0\\ 0&0&-4&0 \\ 0&0&0&-4
     \end{array} \right)$ and
     $N_{4,d} =  \left(
     \begin{array}{rrrr}
     -4d &0&0&0\\ 0&-4&0&0\\ 0&0&-4&0 \\ 0&0&0&-4
     \end{array} \right)$

     \end{center}
     We denote a basis of $M_{2,d}$ and $N_{5,d}$ for the given matrix presentation by $\bf{f}_{1}, \bf{f}_{2}$.
     We also denote a basis of $M_{3,d}$ and $N_{4,d}$ for the given matrix presentation by $\bf{f} _{1}, \bf{f}_{2}, \bf{f}_{3}, \bf{f}_{4}$.
     We give a primitive embedding of each lattice into $U(2) \oplus E_{8}(2)$ as follow.
     \begin{center}
     $\nu : M_{2,d} \hookrightarrow U(2) \oplus E_{8}(2)$, $\nu (\bf{f} _{1}) = \bf{x} + d\bf{y}, \nu(\bf{f} _{2}) = \bf{e} _{1}$\\[0.1cm]
     \end{center}
     \begin{center}
     $\nu : N_{5,d} \hookrightarrow U(2) \oplus E_{8}(2)$, $\nu (\bf{f} _{1}) = \bf{x} - d\bf{y}, \nu(\bf{f} _{2}) = \bf{e} _{1}$
     \end{center}
      \begin{center}
     $\nu : M_{3,d} \hookrightarrow U(2) \oplus E_{8}(2)$, $\nu (\bf{f} _{1}) = \bf{x} + d\bf{y}, \nu(\bf{f} _{2}) = \bf{e} _{1}, \nu(\bf{f}_{3}) = \bf{e} _{3}, \nu(\bf{f}_{4})= \bf{e}_{6}$
     \end{center}
      \begin{center}
    $\nu : N_{4,d} \hookrightarrow U(2) \oplus E_{8}(2)$, $\nu (\bf{f} _{1}) = \bf{x} - d\bf{y}, \nu(\bf{f} _{2}) = \bf{e} _{1}, \nu(\bf{f}_{3}) = \bf{e} _{3}, \nu(\bf{f}_{4})= \bf{e}_{6}$

     \end{center}
     We denote the orthogonal complements of $M_{2,d}$ and $M_{3,d}$, $N_{4,d}$ and $N_{5,d}$ in $U(2) \oplus E_{8}(2)$ by $N_{2,d}$, $N_{3,d}$, $M_{4,d}$ and $M_{5,d}$ respectively.
     For each $\sigma$ and $d$, $N _{\sigma , d}$ is negative definite of rank $12- 2\sigma$ and has no vector
     of self intersection -2. We can check that in each case, the discriminant group of $N_{\sigma , d}$ is
     $$l(N_{2 ,d}) = (\Z / 4d) \oplus (\Z/4) \oplus (\Z/2) ^{6},$$
     $$l(N_{3 ,d}) = (\Z / 4d) \oplus (\Z/4)^{3} \oplus (\Z/2) ^{2},$$
     $$l(N_{4 ,d}) = (\Z / 4d) \oplus (\Z/4)^{3}\textrm{ or}$$
     $$l(N_{5 ,d}) = (\Z / 4d) \oplus (\Z/4).$$
     $U \oplus M_{\sigma , d}$ is an even lattice of signature $(2,2\sigma -2)$. Hence an embedding $U \oplus M_{\sigma , d} \hookrightarrow \Lambda$ is unique up to
     isometry and there exist a K3 surface $X_{\sigma , d}$ over $\C$ such that the transcendental lattice of $X _{\sigma , d}$, $T(X_{\sigma , d})$ is isomorphic
     to $U \oplus M_{\sigma , d}$. (\cite{Mo}, p.112) In this case, the Neron-Severi group of $X_{\sigma , d}$ is an over lattice of $N_{\sigma , d} \oplus \Gamma (2)$.
     For the discriminant of the transcendental lattice of $X_{\sigma , d}$, we can check
     $$d(T(X_{2,d})) = 4^{2}d, \ d(T(X_{3,d})) = 4^{4}d,$$
     $$d(T(X_{4,d})) = d(T(X_{5,d})) = 4^{5}d.$$
     Hence we have
      $$d(NS(X_{2,d})) = -4^{2}d, \ d(NS(X_{3,d})) = -4^{4}d,$$
     $$d(NS(X_{4,d})) = d(NS(X_{5,d})) = -4^{5}d.$$
     \begin{lem}
     If $p=11$ or $p\geq 19$, there exists a positive integer $d < p/8$ such that $-d$ is a non-square modulo $p$.
     If $p \geq 13$ and $p\neq 23$, there exists a positive integer $d <p/8$ such that $-d$ is a square modulo $p$.
     \end{lem}
     \begin{proof}
     It is straightforward.
     \end{proof}
     From now on we assume $0<d < p/8$ and $-d$ is a square modulo $p$ if $\sigma =2$ or $4$ and $-d$ is a non-square modulo $p$ if $\sigma =3$ or 5.
     Because $NS(X _{\sigma ,d}) = -4^{a} d$ for some positive integer $a$ and the rank of $NS(X_{\sigma, d})$ is $22-2\sigma$, by proposition \ref{SH2}, there exists a primitive embedding
     $$i : NS(X_{\sigma , d}) \hookrightarrow \Lambda _{p,\sigma}. $$
     Let $K$ be the orthogonal complement of $i$. Put $ j : \Gamma (2) \hookrightarrow NS(X_{\sigma , d})$ is the given embedding.
     Let $L$ be the orthogonal complement of
     $$ i \circ j : \Gamma (2) \hookrightarrow \Lambda _{p, \sigma }.$$
     Since the discriminant of $NS(X_{\sigma , d})$ is coprime to the discriminant of $\Lambda _{p, \sigma}$, by proposition \ref{SH1},
     $$l(K) = (\Z /p) ^{2\sigma} \oplus l (NS(X_{\sigma , d})).$$
     In particular, the intersection form on $K$ is divisible by $p$. $L$ is an over lattice of $K \oplus N _{\sigma , d}$. Two embeddings
     $K \hookrightarrow L$ and $N_{\sigma , d} \hookrightarrow L$ are primitive.
     By theorem \ref{th1}, it is enough to show that $L$ has no vector of self intersection -2.
     The finite group $L / (K \oplus N_{\sigma , d})$ is a subgroup of $l(N_{\sigma , d})$. Hence for $v \in L$, $4dv \in K \oplus N_{\sigma , d}$.
     Put $4dv = u+w$ where $u \in K$ and $w \in N_{\sigma , d}$. Because $N_{\sigma , d}$ does not have a vector of self intersection -2, we may assume
     $u \neq 0$.
      Since
      $$(u, u) = (4dv,u)=4d(v,u),$$
      $4d$ divides $(u,u)$. Since the intersection form on $K$
     is divisible by $p$, $4dp$ divides $(u,u)$. Therefore we have
     $$(v,v) \leq \frac{1}{(4d)^{2}} (u,u) <-2$$
     and the proof is completed.
          \end{proof}
     \begin{cor}
     Let $k$ be an algebraically closed field of characteristic $p=19$ or $p>23$ and $X$ be a supersingular K3 surface over $k$. $X$ is an Enriques K3 surface
     if and only if the Artin invariant of $X$ is less than 6.
     \end{cor}
     \begin{proof}
     It follows from proposition \ref{J1} and theorem \ref{th2}.
     \end{proof}
     The following result is an analogy for the supersingular case of the main theorem of \cite{Ke}.
     \begin{cor}
     Let $k$ be an algebraically closed field of characteristic $p=19$ or $p>23$ and $X$ be a supersingular Kummer surface over $k$. Then $X$ is an Enriques K3 surface.
     \end{cor}
     \begin{proof}
     A supersingular K3 surface over a field of odd characteristic is a Kummer surface if and only if the Artin invariant is 1 or 2.
     \end{proof}
     \begin{rem}
     In theorem \ref{th2}, we miss only finitely many cases. We could not find an effective way to apply theorem \ref{th1}
     to the cases of small characteristic. However, since the supersingular K3 surface of Artin invariant 1 has an Enriques involution
     over any odd characteristic, considering remark \ref{rem}, it is expected that theorem \ref{th2} holds over any odd characteristic.
     \end{rem}

\vskip 1cm

\noindent
J.Jang\\
Department of Mathematics\\
University of Ulsan \\
Daehakro 93, Namgu Ulsan 680-749, Korea\\ \\
jmjang@ulsan.ac.kr

     \end{document}